\UseRawInputEncoding
\documentclass[10pt]{article}
\usepackage{relsize}
\usepackage[dvips]{color}
\usepackage{epsfig}
\usepackage{url}
\usepackage{float,amsthm,amssymb,amsfonts}
\usepackage{amsmath,graphicx,  latexsym}
\usepackage{xcolor}
\begin{document}
\theoremstyle{plain}
\newtheorem{theorem}{{\bf Theorem}}[section]
\newtheorem{corollary}[theorem]{Corollary}
\newtheorem{lemma}[theorem]{Lemma}
\newtheorem{proposition}[theorem]{Proposition}
\newtheorem{remark}[theorem]{Remark}

\theoremstyle{definition}
\newtheorem{defn}{Definition}
\newtheorem{definition}[theorem]{Definition}
\newtheorem{example}[theorem]{Example}
\newtheorem{conjecture}[theorem]{Conjecture}

\def\im{\mathop{\rm Im}\nolimits}
\def\dom{\mathop{\rm Dom}\nolimits}
\def\rank{\mathop{\rm rank}\nolimits}
\def\nullset{\mbox{\O}}
\def\ker{\mathop{\rm ker}\nolimits}
\def\implies{\; \Longrightarrow \;}

\def\GR{{\cal R}}
\def\GL{{\cal L}}
\def\GH{{\cal H}}
\def\GD{{\cal D}}
\def\GJ{{\cal J}}

\def\set#1{\{ #1\} }
\def\z{\set{0}}
\def\Sing{{\rm Sing}_n}
\def\nullset{\mbox{\O}}

\title{The monoid of monotone and decreasing partial transformations on a finite chain}
\author{\bf  Muhammad Mansur Zubairu\footnote{Corresponding Author: \emph{mmzubairu.mth@buk.edu.ng}}, Abdullahi Umar and  Fatma Salim Al-Kharousi   \\
\it\small  Department of Mathematics, Bayero  University Kano, P. M. B. 3011, Kano, Nigeria\\
\it\small  \texttt{mmzubairu.mth@buk.edu.ng}\\[3mm]
\it\small Department of Mathematical Sciences,\\
\it\small Khalifa University, P. O. Box 127788, Sas al Nakhl, Abu Dhabi, UAE\\
\it\small  \texttt{abdullahi.umar@ku.ac.ae}\\[3mm]
\it\small  Department of Mathematics,\\
\it\small College of Science,\\
\it\small Sultan Qaboos University.\\
\it\small \texttt{fatma9@squ.edu.om}\\
\it\small \texttt{The results in this paper were obtained in October 2024}.}
\date{\today}
\maketitle\

\begin{abstract} In this article, we consider  the monoid of all monotone  and order-decreasing partial transformations denoted as $\mathcal{DORP}_{n}$ on an $n$ ordered chain $[n]=\{1, \ldots,n\}$, its  two-sided ideal $I(n,p)= \{\rho \in \mathcal{DORP}_{n} : \, |\im \, \rho| \leq  p\}$ and  the  Rees quotient ${RQ}_{p}(n)$ of the ideal $I(n,p)$.  We compute the order of the monoid $\mathcal{DORP}_{n}$ and show that for any semigroup $S$ in $\{\mathcal{DORP}_{n}, \,  I(n,p), \,  {RQ}_{p}(n)\}$, $S$ is  abundant for all values of $n$. In particular, we show that the Rees quotient ${RQ}_{p}(n)$, is a non-regular $0-*$bisimple abundant semigroup. In addition, we compute the ranks of the Rees quotient ${RQ}_{p}(n)$ and the two-sided ideal $I(n,p)$. Finally, the rank of $\mathcal{DORP}_{n}$ is determined to be $3n-2$. \end{abstract}


\section{Introduction}

Let $[n]$ denote the finite $n-$chain $\{1, 2, \dots, n\}$. A function $\rho$ with  domain and range both being subsets of $[n]$ is termed a \emph{partial transformation} of $[n]$, and it is said to be \emph{full} (or \emph{total}) if its domain is the whole of $[n]$. The set consisting of all partial transformations on $[n]$ is usually denoted by $\mathcal{P}_{n}$ and is known as \emph{the semigroup of all partial transformations} on $[n]$, more commonly referred to as \emph{the partial symmetric monoid}. A transformation $\rho \in \mathcal{P}_{n}$ is called an \emph{isotone} (or order preserving) function (resp., an \emph{antitone} (order reversing) function) if (for all $x,y \in \dom\,\rho$) $x \leq y$ implies $x\rho \leq y\rho$ (resp., $x\rho \geq y\rho$); it is called \emph{monotone} if it is either isotone or antitone or both; and it is called \emph{order decreasing} if (for all $a \in \dom \, \rho$) $a\rho \leq a$. Let $\mathcal{POD}_{n}$ be \emph{the monoid of all monotone partial transformations} on $[n]$. This monoid first appeared in Fernandes \cite{vm} and subsequently in Umar \cite{auc} and East \emph{et. al.}, \cite{east}, where its congruence/rank properties and combinatorial properties, and its maximal subsemigroups were studied, respectively. The algebraic properties of various subsemigroups of monotone maps in various classes have been investigated over the years, for example see \cite{bugay, bugay1, il}.

In line with \cite{auc},  $\mathcal{LS}_{n}$ (\emph{the semigroup of all isotone and order-decreasing partial transformations} on $[n]$) is  refer to as the \emph{large} \emph{Schr\"{o}der} monoid, which is defined as:

\begin{equation}\label{qn111}
\mathcal{LS}_{n} = \mathcal{OP}_n \cap \mathcal{DP}_n,
\end{equation}

\noindent where  $\mathcal{OP}_n$ and $\mathcal{DP}_n$ denote  the \emph{semigroup of all isotone partial transformations} on $[n]$  and \emph{the semigroup of all order-decreasing partial transformations} on $[n]$, respectively. The order of this monoid, obtained in \cite{al3}, corresponds to the \emph{double} (or \emph{large}) \emph{Schr\"{o}der} number:

\begin{equation}\label{nums}
s_{0}=1, \quad s_{n}= \frac{1}{n+1} \sum\limits_{r=0}^{n}\binom{n+1}{n-r}\binom{n+r}{r}, \quad n\geq 1.
\end{equation}

 Now, let
\begin{equation}\label{qn1}
 	\mathcal{DORP}_{n} = \{\rho \in \mathcal{POD}_{n} : \,   \rho \text{ is \it decreasing} \}
\end{equation}
\noindent be the set of all decreasing maps in $\mathcal{POD}_{n}$. This set can equivalently be expressed as:
\begin{align*}
\mathcal{DORP}_{n} &= \{\rho \in \mathcal{P}_{n} : \,   \rho \text{ is {\it monotone} and \it decreasing} \} \\
&= \{\rho \in \mathcal{P}_{n} : \,   \rho \text{ is {\it isotone} and \it decreasing} \} \cup \{\rho \in \mathcal{P}_{n} : \,   \rho \text{ is {\it antitone} and \it decreasing} \}.
\end{align*}
\noindent Therefore, if we let $DRP_{n} = \{\rho \in \mathcal{P}_{n} : \,   \rho  \text{ is {\it antitone} and \it decreasing}\},$ then
\begin{equation}\label{qn2}
	\mathcal{DORP}_{n} = \mathcal{LS}_{n} \cup DRP_{n}.
\end{equation}
Moreover, the set $\mathcal{DORP}_{n}$ can be expressed as:

$$
\mathcal{DORP}_{n} = \mathcal{POD}_{n} \cap \mathcal{DP}_{n}.
$$
\noindent It is a routine matter to show that $\mathcal{DORP}_{n}$ is a monoid. It shall be refer to  as the \emph{monoid of all monotone and decreasing partial transformations} on $[n]$. There seems to be no prior discussion of the monoid \(\mathcal{DORP}_{n}\) in the literature. This paper investigates certain algebraic features of the monoid \(\mathcal{DORP}_{n}\) and its rank properties.

For  map $\rho\in \mathcal{DORP}_{n}$, we shall adopt the notations $\dom \rho$, $1_{[n]}$, $b(\rho) = |\dom \, \rho|$, $\im \rho$,  $h(\rho) = |\im \, \rho|$, and $F(\rho) = \{x \in \dom \, \rho : x\rho = x\}$ to denote the domain set of  $\rho$, the identity mapping on $[n]$, the width of $\rho$, the image set of $\rho$, the height of $\rho$,  and the set of fixed points of $\rho$, respectively. Additionally, we will let $f(\rho) = |F(\rho)|$ denote the number of fixed points of $\rho$.  A subset $X$ of $[n]$ is said to be \emph{convex} if, for any $ x, y\in X $ such that $x \leq y$, and for any $c \in [n]$, if $x < c < y$, then $c \in X$. We shall be using the notations $E(S)$ and $\textbf{0}$ to denote the set of idempotents  and the zero element of a semigroup $S$, respectively.

Furthermore, we shall adopt the right-hand composition of two transformations, say $\rho$ and $\sigma$ in $\mathcal{P}_{n}$, defined as
$$
x(\rho \circ \sigma) = ((x)\rho) \sigma
$$
\noindent for all $x \in \dom\, \rho$. To be concise, we will denote \(\rho \sigma\) as \(\rho \circ \sigma\).

  Moreover, for $0\le p\le n-1$,  let \begin{equation} \label{kn} I(n, \,p)=\{\rho\in  \mathcal{DORP}_{n}: \, |\im \, \rho|\le p\}\end{equation}
  \noindent be the two-sided ideal of $\mathcal{DORP}_{n}$, which consist all decreasing and monotone transformations in  $\mathcal{DORP}_{n}$, each with a height not more than $p$.

  Additionally, given that $p\geq 1$,  denote \begin{equation}\label{knn} {RQ}_{p}(n)= I(n, \,p)/ I(n, \, p-1)  \end{equation}
  \noindent to be the Rees quotient semigroup of $I(n, \,p)$, where the elements in ${RQ}_{p}(n)$ can be regarded as the elements of $\mathcal{DORP}_{n}$ that possess a height of precisely $p$. If the product of two elements from ${RQ}_{p}(n)$ has a height that does not reach $p$, then the result is $\textbf{0}$; otherwise, it is an element of ${RQ}_{p}(n) \setminus \{\textbf{0}\}$.

\indent According to \cite{HRS}, every $\rho \in \mathcal{LS}_{n}$ can be represented in two-line notation as

\begin{equation}
\label{1}
\rho = \begin{pmatrix} A_1 & \dots & A_p \\ a_1 & \dots & a_p \end{pmatrix} \quad (1 \leq p \leq n),
\end{equation}

\noindent where $a_{i} \leq \min A_{i}$ for all $i\in\{1,\ldots, p\}$ since $\rho$ is a decreasing transformation, and each of the sets $A_i$ $(1 \leq i \leq p)$ denotes an equivalence class determined by the following relation:

$$\textnormal{ker } \rho = \{(x, y) \in \dom \, \rho \times \dom \, \rho : x\rho = y\rho\}.$$

\noindent The collection of these equivalence classes shall be denoted as $\textnormal{\bf Ker } \rho = \{A_1,  \dots, A_p\}$. In addition, the kernel $\textnormal{\bf Ker } \rho$ is linearly ordered (meaning that for any indices $i < j$, the relation $A_{i} < A_{j}$ is equivalent to $x < y$ for all $x \in A_{i}$ and $y \in A_{j}$). Furthermore, we can assume without loss of generality that $1 \leq a_{1} < \dots < a_{p} \leq n$. For a comprehensive introduction to the basic ideas of semigroup theory, we recommend the texts by Howie \cite{howi} and Higgins \cite{ph}.

This paragraph offers a concise explanation of the framework of the paper. Section 1 provide definitions of some basic terms and compute the order of the monoid \( \mathcal{DORP}_{n} \). In Section 2, we characterize of all the Green's relations along with their starred counterparts in the monoid \( \mathcal{DORP}_{n} \) and its two-sided ideal \( I(n, \,p) \). Moreover, we demonstrate that the monoid \( \mathcal{DORP}_{n} \) and its two-sided ideals are abundant semigroups. Finally, in Section 3, we calculate the rank of the Rees quotient semigroup \( RQ_{n}(p) \), the two-sided ideal \( I(n, \,p) \), and the monoid \( \mathcal{DORP}_{n} \). The computation of the rank of this monoid and its two-sided ideals differs from the conventional understanding associated with certain known monoids, where most arguments or generating sets are derive from elements of height \(n-1\). However, in the monoid \( \mathcal{DORP}_{n} \), elements of height \(n-1\) are insufficient to generate the structure; additional elements below this height are necessary, as will be discussed in Section 3.

Now let us briefly examine the elements in $DRP_{n}$. First, notice that every element of height 1 in $\mathcal{DORP}_{n}$ is both isotone and antitone; that is, it belongs to both $\mathcal{LS}_{n}$ and $DRP_{n}$. However, antitones of height greater than or equal to 2 cannot be isotones. Therefore, we have the following remark.

\begin{remark}\label{rmk1}
Every element in $DRP_{n}$ of height $1 < p < n$ of the form
\begin{equation}\label{2}
\rho = \begin{pmatrix} A_1 & \dots & A_p \\ a_p & \dots & a_1 \end{pmatrix},
\end{equation}
 has the following properties:
\begin{itemize}
  \item[(i)] $1 \leq a_{1} < \cdots < a_{p} \leq \min A_{1}$;
  \item[(ii)] Note that when multiplying two isotone maps or two antitone maps, the result remains isotone. In contrast, the product of an isotone map with an antitone map yields an antitone map. Similarly, when an antitone map is multiplied by an isotone map, the outcome is antitone as well.

  \item[(iii)] It should also be emphasized that all elements in $\mathcal{DORP}_{n}$ of height greater than $\lceil \frac{n}{2} \rceil$ is necessarily isotone.
\end{itemize}
\end{remark}

Consequently, every element in $DRP_{n}$ has height $p$ that ranges from $1 \leq p \leq \lceil \frac{n}{2} \rceil$; thus, we have the following lemma.

\begin{lemma}\label{rev}
An element $\rho \in \mathcal{LS}_{n}$, as expressed in \eqref{1}, is reversible; that is, $\begin{pmatrix} A_1 & \dots & A_p \\ a_p & \dots & a_1 \end{pmatrix} \in \mathcal{DORP}_{n}$ if and only if $a_{p} \leq \min A_{1}$ and $1 \leq p \leq \lceil \frac{n}{2} \rceil$.
\end{lemma}
\begin{proof} Let $w=\min (\dom \, \rho) =\min A_{1}\geq a_{p}$, by the order decreasing property. However, notice that $n-p+1\geq w $ and $a_{p}\geq p$,  and so,  $n-p+1\geq w\geq a_{p}\geq p$, which implies $n+1\geq 2p$. Hence $p\leq \lfloor \frac{n+1}{2}\rfloor=\lceil\frac{n}{2}\rceil$. The converse is trivial.

\end{proof}
We note the following well known combinatorial identity, which is useful in our subsequent discussions.

 \begin{lemma}[\cite{44}, (3b), p.8]\label{id}
For all n \( m, n, k\in \mathbb{N} \), it follows that
\[
\sum_{j=k}^{m} \binom{j}{k} \binom{m+n-j}{n} = \binom{m+n+1}{n+k+1}.
\]
\end{lemma}

 It is now clear that to obtain the size of $\mathcal{DORP}_{n}$, it is sufficient to compute the size of the set $DRP_{n}$. For a transformation $\rho$ on a finite chain, let $b(\rho) = r$ and $h(\rho) = p$. Then, $p \leq r \leq n$. Define the combinatorial functions:

\[F(n,r,p)=|\{\rho\in {DRP}_n : \, b(\rho)=r \text{ and } h( \rho)=p\}|\]
\noindent and
\[F(n,p)=|\{\rho\in {DRP}_n : \,  h(\rho)=p\}|.\]

Then we present the following theorem.
\begin{theorem}
The number of elements in ${DRP}_n$ of a fixed width $r$ and height $p$ is $F(n,r,p)= {r-1 \choose p-1} {n+1 \choose r+p} $
\end{theorem}
\begin{proof}

 Let $\rho\in{DRP}_n$ be as expressed in \eqref{2}, where $b(\rho)=r$ and $h(\rho)=p$. Moreover, let $w=\min(\dom \, \rho)$. Note that $p\leq w\leq n - p + 1$, from the proof of Lemma \ref{rev},   and for all $x_i \in \dom \, \rho$ and $a_j \in \im \, \rho$, we have $ a_j\leq x_i$. To count the number of $\rho\in{DRP}_n$, we first choose the domain elements. Now, since $w\in \dom \, \rho$, we can choose the remaining $r-1$ elements from $[n]\setminus\{1,\ldots,w\}$, i.e., in ${n - w \choose r - 1}$ ways. Next, we partition the $r$ elements in the domain into $p$  convex (modulo $\dom \, \rho$) blocks in ${r - 1 \choose p - 1}$ ways. Then, we can choose the $p$ images from the set $\{1,\ldots,w\}$  in ${w \choose p}$ ways. Finally, taking the sum of the product: ${n-w\choose r-1} {r-1\choose p-1} {w\choose p}$, from $w=p$ to $w=n-p+1$ produces
\begin{align*}  F(n,r,p)&=\sum_{w=p}^{n-p+1} {n-w\choose r-1} {r-1\choose p-1} {w\choose p}\\&=
{r-1\choose p-1}\sum_{w=p}^{n-p+1} {n-w\choose r-1}  {w\choose p}
\\&=
{r-1 \choose p-1} {n+1 \choose r+p}, \quad (\text{by Lemma \ref{id}})
\end{align*}
\noindent as postulated.

\end{proof}
\bigskip

\begin{corollary}\label{co}

The number of elements in ${DRP}_n$ of a fixed height $1\le p\le \lceil\frac{n}{2}\rceil$ is
\begin{equation}\label{coo} F(n,p)= \sum_{r=p}^{n-p+1}{r-1 \choose p-1} {n+1 \choose r+p}.\end{equation}
\end{corollary}
\begin{proof}
 We get the number by varying the width within the possible range which is $p\leq r\leq n-p+1$.
\end{proof}

We will now state the following lemma.
\begin{lemma} Let $F(n,p)$ be as in  \eqref{coo} and let $a_{n} = \sum\limits_{p=2}^{\lceil\frac{n}{2} \rceil} F(n, p)$. Then
\[
a_{n} =  \sum\limits_{p=2}^{\lceil\frac{n}{2} \rceil}  \sum_{r=p}^{n-p+1}  {r-1 \choose p-1} {n+1 \choose r+p}.
\]
\end{lemma}

At this point, we present the following result.

\begin{theorem} Let $s_{n}$ be as defined in \eqref{nums}. Then $|\mathcal{DORP}_{n}|=s_{n}+a_{n}$.
\end{theorem}
\begin{proof} Notice that, as in \eqref{qn2}, $\mathcal{DORP}_{n} = \mathcal{LS}_{n} \cup DRP_{n}$. However, elements of exactly height $1$ are both isotone and antitone maps, so  $$|\mathcal{LS}_{n} \cap DRP_{n}|=F(n,1).$$ \noindent Thus,  $$|\mathcal{DORP}_{n}| =|\mathcal{LS}_{n} \cup DRP_{n}|= |\mathcal{LS}_{n}|+| DRP_{n}|-| \mathcal{LS}_{n} \cap DRP_{n}|,$$ \noindent and therefore
\[|\mathcal{DORP}_{n}| =s_{n}+\sum\limits_{p=1}^{\lceil\frac{n}{2} \rceil} {F(n, p)}-F(n,1)=s_{n}+ F(n,1)+\sum\limits_{p=2}^{\lceil\frac{n}{2} \rceil} {F(n, p)}-F(n,1)=s_{n}+a_{n},\] as required.
\end{proof}

\begin{remark}
 It is  significant to highlight that the triangle of numbers $F(n,\, p)$ and the sequences $a_{n}$ and $s_{n}+a_{n}$ are not yet recorded in The Online Encyclopedia of Integer Sequences \cite{slone}.
\end{remark}

\section{Green's  and starred Green's relations}

 The five Green's relations defined on a semigroup that usually contains idempotents are $\mathcal{L}$, $\mathcal{R}$, $\mathcal{D}$, $\mathcal{J}$, and $\mathcal{H}$. These relations are defined as follows: for $g$, $h$ in a semigroup $S$, $(g,h) \in \mathcal{L}$ if and only if $S^{1}g = S^{1}h$; $(g,h) \in \mathcal{R}$ if and only if $gS^{1} = gS^{1}$; $(g,h) \in \mathcal{J}$ if and only if $S^{1}gS^{1} = S^{1}hS^{1}$; while the relation $\mathcal{D}$ is the ``join" of $\mathcal{R}$ and $\mathcal{L}$, and $\mathcal{H} = \mathcal{L} \cap \mathcal{R}$. However, if $S$ is finite, the relations $\mathcal{J}$ and   $\mathcal{D}$ are the same, as shown in [\cite{howi}, Proposition 2.1.4]. As a result, we will focus entirely on the characterization of the relations $\mathcal{L}$, $\mathcal{R}$, $\mathcal{D}$, and $\mathcal{H}$  within $\mathcal{DORP}_{n}$. An element $x$ in a semigroup $S$ is said to be \emph{regular} if $x = xyx$ for some $y\in S$, and $S$ is said to be a \emph{regular semigroup} if every element in $S$ is regular. The algebraic properties of a regular semigroup have received significant attention in the literature; see, for example, sections 2 and 4 of Howie's book \cite{howi}.

In this section moving forward, we will refer to the maps $\rho$ and $\sigma$ within $\mathcal{DORP}_{n}$ as:
 \begin{equation} \label{eqq3}
	\rho = \begin{pmatrix}A_1&\dots& A_p\\a_{1}&\dots& a_p\end{pmatrix} \text{and} \  \sigma = \begin{pmatrix} B_1 &  \dots & B_p \\ b_{1} & \dots  & b_p \end{pmatrix}  \, (1\leq p\leq n).
\end{equation}

Next, we have the following theorems to consider.
\begin{theorem}\label{l}
Let $\rho, \sigma \in \mathcal{DORP}_{n}$ be as in \eqref{eqq3}. Then $\rho \mathcal{L} \sigma$ if and only if $a_i = b_i$ and $\min a_i \rho^{-1} = \min b_i \sigma^{-1}$ for all $1 \leq i \leq p$.
\end{theorem}

\begin{proof}
Notice that $\mathcal{DORP}_{n} = \mathcal{POD}_{n} \cap \mathcal{DP}_{n}$. Thus, the result follows from [\cite{vm}, Proposition 1.1(2)] and [\cite{umar}, Theorem 2.4.4(3)].
\end{proof}

\begin{theorem}\label{r}
Let $\mathcal{DORP}_{n}$ be as defined in \eqref{qn2}. Then $\mathcal{DORP}_{n}$ is $\mathcal{R}$-trivial.
\end{theorem}

\begin{proof}
It is clear that $\mathcal{DORP}_{n}$ is $\mathcal{R}$-trivial, being a subsemigroup of an $\mathcal{R}$-trivial semigroup, $\mathcal{DP}_{n}$ [\cite{umar}, Theorem 2.4.4(2)].
\end{proof}

 As a result of Theorems \ref{l} and \ref{r}, we can directly derive the following corollaries.
\begin{corollary}
On the monoid $\mathcal{DORP}_{n}$, $\mathcal{H} = \mathcal{R}$.
\end{corollary}

 \begin{corollary}\label{rem1} Let $\rho \in  \mathcal{DORP}_{n}$. Then $\rho$ is regular if and only if $\rho$ is an idempotent. As a result, the monoid $\mathcal{DORP}_{n}$ is non-regular for all $n \geq 2$.
 \end{corollary}
\begin{proof} The conclusion arises from the observation that if a semigroup is $\mathcal{R}$-trivial, then no non-idempotent element is regular.
 \end{proof}

 \begin{theorem} Let $\mathcal{DORP}_{n}$ be as defined in \eqref{qn2}. Then $ \mathcal{D} = \mathcal{L}$.	
	\end{theorem}
\begin{proof}
The conclusion follows from the observation  that $\mathcal{L}\subseteq\mathcal{D}$, and that $\mathcal{DORP}_{n}$ is $\mathcal{R}$-trivial from Theorem \ref{r}.
\end{proof}

Consequently, utilizing the three previously discussed theorems, we can deduce the following characterization of Green's equivalences within the semigroup $S$ in $\{RQ_{p}(n), I(n, \,p)\}$.
\begin{theorem}
Let $S \in \{RQ_{p}(n), I(n, \,p)\}$ and let $\sigma ,\rho \in S$. Then
\begin{itemize}
  \item[(a)] ($\rho, \sigma)\in \mathcal{L} $ if and only if $a_i = b_i$ and $\min a_i \rho^{-1} = \min b_i \sigma^{-1}$ for all $i\in \{1,\ldots, p\}$;
  \item[(b)] The semigroup $S$ is  $\mathcal{R}$-trivial;
  \item[(c)] $\mathcal{L} = \mathcal{D}$;
  \item[(d)] $\mathcal{R} = \mathcal{H}$.
\end{itemize}
\noindent Consequently, the semigroup $S$ is non-regular  for $p \geq 2$.
\end{theorem}

Now, since $\mathcal{DORP}_{n}$ is not a regular semigroup,  it is standard practice to investigate the starred Green's equivalences to determine its algebraic classification. As a result, we will continue by characterizing the starred versions of Green's equivalences on $S\in \{\mathcal{DORP}_{n}, {RQ}_{p}(n), \, I(n, \,p) \}$. For the definitions and basic properties of these relations, we refer the reader  to Fountain \cite{FOUN2}.

 Similar to the Green's relations, there are five starred Green's equivalences, they are: $\mathcal{L}^*$, $\mathcal{R}^*$, $\mathcal{D}^*$, $\mathcal{H}^*$, and $\mathcal{J}^*$. The relation $\mathcal{H}^*=\mathcal{L}^*\cap\mathcal{R}^*$, but $\mathcal{D}^*$ is the join of $\mathcal{L}^*$ and $\mathcal{R}^*$. It is a known fact that in a finite non-regular semigroup, $\mathcal{L}^* \circ \mathcal{R}^*$ does not necessarily commute. Moreover, the relations $\mathcal{L}^*$ and $\mathcal{R}^*$ have the following characterizations on any semigroup $S$:

\begin{equation}\label{l8}
    \mathcal{L}^* = \{(a,b) \in S \times S : (\text{for all } x,y \in S^1) \ ax = ay \iff bx = by\};
\end{equation}

\begin{equation}\label{r8}
    \mathcal{R}^* = \{(a,b) \in S \times S : (\text{for all } x,y \in S^1) \ xa = ya \iff xb = yb\}.
\end{equation}

 A non-regular semigroup $S$ is called \emph{left abundant} if each of its $\mathcal{L}^*$-class contains an idempotent; it is  \emph{right abundant} if each of its $\mathcal{R}^*$-class contains an idempotent; and it is \emph{abundant} if it is both left and right abundant. Abundant semigroups were first introduced by Fountain \cite{FOUN, FOUN2}. Several categories of transformation semigroups have been identified as being either left abundant, right abundant, or abundant, see, for example, \cite{ al1, um, umar, quasi, zm1, zuf}. The following definition and lemmas from \cite{quasi} are useful for our subsequent discussions: A subsemigroup $A$ of $S$ is said to be  an \emph{inverse-ideal} of a semigroup $S$ if for all $a \in A$,  $aa^{\prime}a = a$ for some $a^{\prime} \in S$  and both $a^{\prime}a$ and $aa^{\prime}$ are in $A$.

 \begin{lemma}[\cite{quasi}, Lemma 3.1.8.]\label{inv1}  For a semigroup $S$, every inverse-ideal $A$ of  $S$ is abundant.
 \end{lemma}

 \begin{lemma} [\cite{quasi}, Lemma 3.1.9.]  \label{inv2}  If $A$ is an inverse-ideal of a semigroup $S$, then \begin{itemize} \item[(a)]  $\mathcal{L}^{*} (A) = \mathcal{L}(S) \cap (A \times A)$; \item[(b)] $\mathcal{R}^{*}( A) = \mathcal{R}(S) \cap(A \times A)$; \item[(c)] $\mathcal{H}^{*}( A) = \mathcal{H}(S) \times (A \times A).$\end{itemize}
 \end{lemma}

The next result is now at our disposal.
 \begin{theorem}\label{inv} Let \(\mathcal{DORP}_{n}\)  be as defined in \eqref{qn2}. Then, for $n\ge 2$, \(\mathcal{DORP}_{n}\) is an  inverse-ideal of  $\mathcal{P}_{n}$.
 \end{theorem}
 \begin{proof} Let $\rho\in \mathcal{DORP}_{n}$. Then $\rho$ is either an isotone map or an antitone map.

  \noindent \textbf{Case (i.)} Suppose $\rho$ is an isotone map  as expressed in \eqref{1}, and let $c_{i}=\min A_{i}$ for all $i\in\{1,\ldots,p\}$. Now define $\rho^{\prime}$ as: \[\rho^{\prime}=\begin{pmatrix}
a_1  & \dots & a_p\\
c_1   & \dots & c_p
\end{pmatrix} .\]
\noindent Clearly, $\rho^{\prime}$ is in $\mathcal{P}_{n}$. Notice that:

\begin{align*}\rho\rho^{\prime}\rho &=\begin{pmatrix}
A_1  & \dots & A_p\\
a_1   & \dots & a_p
\end{pmatrix}\begin{pmatrix}
a_1  & \dots & a_p\\
c_1   & \dots & c_p
\end{pmatrix}\begin{pmatrix}
A_1  & \dots & A_p\\
a_1   & \dots & a_p
\end{pmatrix}\\&= \begin{pmatrix}
A_1  & \dots & A_p\\
a_1   & \dots & a_p
\end{pmatrix}=\rho. \end{align*}

Additionally, \[\rho\rho^{\prime}=\begin{pmatrix}
A_1  & \dots & A_p\\
a_1   & \dots & a_p
\end{pmatrix}\begin{pmatrix}
a_1  & \dots & a_p\\
c_1   & \dots & c_p
\end{pmatrix}
=\begin{pmatrix}
A_1  & \dots & A_p\\
c_1   & \dots & c_p
\end{pmatrix}\in E(\mathcal{DORP}_{n})\subset \mathcal{DORP}_{n}.\]

\noindent Moreover, \[\rho^{\prime}\rho=\begin{pmatrix}
a_1  & \dots & a_p\\
c_1   & \dots & c_p
\end{pmatrix}\begin{pmatrix}
A_1  & \dots & A_p\\
a_1   & \dots & a_p
\end{pmatrix}=\begin{pmatrix}
a_1  & \dots & a_p\\
a_1   & \dots & a_p
\end{pmatrix}=\text{1}_{\im \, \rho}\in E(\mathcal{DORP}_{n})\subset \mathcal{DORP}_{n} .\]

\noindent \textbf{Case (ii.)} Suppose $\rho$ is an antitone map be  expressed as \[\rho= \begin{pmatrix}
A_1  & \dots & A_p\\
a_p   & \dots & a_1
\end{pmatrix},\] \noindent and let $c_{i}=\min A_{i}$ ($1\leq i\leq p$). Now define
 \[\rho^{\prime}=\begin{pmatrix}
a_1  & \dots & a_p\\
c_p   & \dots & c_1
\end{pmatrix}.\]
\noindent Clearly $\rho^{\prime}\in \mathcal{P}_{n}$, and  observe that:

\begin{align*}\rho\rho^{\prime}\rho &=\begin{pmatrix}
A_1  & \dots & A_p\\
a_p   & \dots & a_1
\end{pmatrix}\begin{pmatrix}
a_1  & \dots & a_p\\
c_p   & \dots & c_1
\end{pmatrix}\begin{pmatrix}
A_1  & \dots & A_p\\
a_p   & \dots & a_1
\end{pmatrix}\\&= \begin{pmatrix}
A_1  & \dots & A_p\\
a_p   & \dots & a_1
\end{pmatrix}=\rho. \end{align*}

Moreover, \[\rho\rho^{\prime}=\begin{pmatrix}
A_1  & \dots & A_p\\
a_p   & \dots & a_1
\end{pmatrix}\begin{pmatrix}
a_1  & \dots & a_p\\
c_p   & \dots & c_1
\end{pmatrix}
=\begin{pmatrix}
A_1  & \dots & A_p\\
c_1   & \dots & c_p
\end{pmatrix}\in E(\mathcal{DORP}_{n})\subset \mathcal{DORP}_{n}.\]

\noindent Furthermore, \[\rho^{\prime}\rho=\begin{pmatrix}
a_1  & \dots & a_p\\
c_p   & \dots & c_1
\end{pmatrix}\begin{pmatrix}
A_1  & \dots & A_p\\
a_p   & \dots & a_1
\end{pmatrix}=\begin{pmatrix}
a_1  & \dots & a_p\\
a_1   & \dots & a_p
\end{pmatrix}=\text{1}_{\im \, \rho}\in E(\mathcal{DORP}_{n})\subset \mathcal{DORP}_{n} .\]
  Thus,  $\mathcal{DORP}_{n}$ is an inverse-ideal of $\mathcal{P}_{n}$, as postulated.
 \end{proof}

 Consequently, we have the following result.

\begin{theorem}
	Let \(\mathcal{DORP}_{n}\)  be as defined in \eqref{qn2}. Then, \(\mathcal{DORP}_{n}\) is  abundant for all  $n\ge 2$.
\end{theorem}

\begin{proof} The conclusion follows from  Lemma \ref{inv1}, and Theorem \ref{inv}.
\end{proof}

 \begin{theorem} \label{a1}
Let \(\mathcal{DORP}_{n}\)  be as defined in \eqref{qn2}. Then,  for $\rho, \sigma\in \mathcal{DORP}_{n}$, we have:
 \begin{itemize}
   \item[(a)] $(\rho, \sigma)\in\mathcal{L}^*$  if and only if $\im  \,  \rho = \im  \, \sigma$;
   \item[(b)] $(\rho, \sigma)\in\mathcal{R}^*$ if and only if $\ker \, \rho = \ker \, \sigma$;
   \item[(c)] $(\rho, \sigma)\in\mathcal{H}^*$ if and only if $\im \, \rho = \im \, \sigma$ and $\ker \, \rho = \ker \, \sigma$;
   \item[(d)] $(\rho, \sigma)\in\mathcal{D}^*$ if and only if $|\im \, \rho| = |\im  \, \sigma|$.
 \end{itemize}
 \end{theorem}

\begin{proof} Clearly, (a), (b), and (c) follow directly from Lemma \ref{inv2} and Theorem \ref{inv}.

(d) Suppose $(\rho, \sigma)\in \mathcal{D}^{*} $. Then, by (\cite{howi}, Proposition 1.5.11), it means that there exist elements $\delta_{1}, \delta_{2}, \dots, \delta_{2n-1}$ in $\mathcal{DORP}_{n}$ such that $\rho \mathcal{L}^{*} \delta_{1}$, $\delta_{1} \mathcal{R}^{*} \delta_{2}$, $\delta_{2} \mathcal{L}^{*} \delta_{3}, \dots, \delta_{2n-1} \mathcal{R}^{*} \sigma$ for some $n \in \mathbb{N}$. As such, from (a) and (b), we immediately deduce  $\im~\rho = \im~\delta_{1}$, $\ker~\delta_{1} = \ker~\delta_{2}$, $\im~\delta_{2} = \im~\delta_{3}, \dots, \ker~\delta_{2n-1} = \ker~\sigma$. So, it follows that $|\im~\rho| = |\im~\delta_{1}| = |\dom~\delta_{1}/\ker~\delta_{1}| = |\dom~\delta_{2}/\ker~\delta_{2}| = \dots = |\dom~\delta_{2n-1}/\ker~\delta_{2n-1}| = |\dom~\sigma/\ker~\sigma| = |\im~\sigma|.$

Conversely, suppose $|\im~\rho|=|\im~\sigma|$. Thus, either $\rho$ and $\sigma$  are isotone maps of the form  \begin{equation*} \rho=\left(\begin{array}{ccc}
                                                                            A_{1}  & \dots & A_{p} \\
                                                                            a_{1} & \dots & a_{p}
                                                                          \end{array}
\right)\text{ and } \sigma=\left(\begin{array}{ccc}
                                                                            B_{1}  & \dots & B_{p} \\
                                                                            b_{1} & \dots & b_{p}
                                                                          \end{array}
\right);\end{equation*}

\noindent or $\rho$ and $\sigma$  are antitone maps of the form  \begin{equation*}\label{3} \rho=\left(\begin{array}{ccc}
                                                                            A_{1}  & \dots & A_{p} \\
                                                                            a_{p} & \dots & a_{1}
                                                                          \end{array}
\right)\text{ and } \sigma=\left(\begin{array}{ccc}
                                                                            B_{1}  & \dots & B_{p} \\
                                                                            b_{p} & \dots & b_{1}
                                                                          \end{array}
\right);\end{equation*}
\noindent or  $\rho$   is isotone and $\sigma$ is antitone  of the form  \begin{equation*} \rho=\left(\begin{array}{ccc}
                                                                            A_{1}  & \dots & A_{p} \\
                                                                            a_{1} & \dots & a_{p}
                                                                          \end{array}
\right)\text{ and } \sigma=\left(\begin{array}{ccc}
                                                                            B_{1}  & \dots & B_{p} \\
                                                                            b_{p} & \dots & b_{1}
                                                                          \end{array}
\right);\end{equation*}

 \noindent or $\rho$   is antitone and $\sigma$ is isotone  of the form  \begin{equation*} \rho=\left(\begin{array}{ccc}
                                                                            A_{1}  & \dots & A_{p} \\
                                                                            a_{p} & \dots & a_{1}
                                                                          \end{array}
\right)\text{ and } \sigma=\left(\begin{array}{ccc}
                                                                            B_{1}  & \dots & B_{p} \\
                                                                            b_{1} & \dots & b_{p}
                                                                          \end{array}
\right).\end{equation*}

\noindent Thus, in either of the cases, we define \begin{equation*} \delta=\left(\begin{array}{ccc}
                                                                            A_{1}  & \dots & A_{p} \\
                                                                            {1} & \dots & {p}
                                                                          \end{array}
\right)\text{ and } \gamma=\left(\begin{array}{ccc}
                                                                            B_{1}  & \dots & B_{p} \\
                                                                            {1} & \dots & {p}
                                                                          \end{array}
\right).\end{equation*}

\noindent It is obvious that  $\gamma$ and $\delta$ are in $\mathcal{DORP}_{n}$. Observe that in either of the cases, $\ker \rho = \ker \delta$, $\im \delta = \im \gamma$, and $\ker \gamma = \ker \sigma$. Thus, from (a) and (b), we conclude that $\rho \mathcal{R}^{*} \delta \mathcal{L}^{*} \gamma \mathcal{R}^{*} \sigma$.

 \noindent Additionally,  define $\delta=\left(\begin{array}{ccc}
                                                                            n-p+{1}  & \dots & n \\
                                                                            a_{1} & \dots & a_{p}
                                                                          \end{array}
\right)$ and  $\gamma=\left(\begin{array}{ccc}
                                                                            n-p+1  & \dots & n \\
                                                                            b_{1} & \dots & b_{p}
                                                                          \end{array}
\right)$. Evidently, $\delta$ and $\gamma$ belong to $\mathcal{DORP}_{n}$. Moreover,  notice that in either of the cases,  $\im \, \rho=\im  \, \delta$,   $\ker \, \delta= \ker \, \gamma$ and  $\im \, \gamma=\im \,  \sigma$.

Therefore, from (a) and (b), we obtain $\rho \mathcal{L}^{*} \delta \mathcal{R}^{*} \gamma \mathcal{L}^{*} \sigma$. Consequently, according to Proposition 1.5.11 in \cite{howi}, it follows that $\rho \mathcal{D}^{*} \sigma$. This completes the proof.

\end{proof}

We can now state the following result.
\begin{theorem}  Let $\mathcal{DORP}_{n}$ be as defined in \eqref{qn111}. For any $\rho\in \mathcal{DORP}_{n}$,
\[|H^{*}_{\rho}| =
\begin{cases}
1, & \text{if }  h(\rho) \in \{0, 1, \lceil \frac{n}{2} \rceil + 1, \ldots, n\}; \\
2, & \text{if }2 \leq h(\rho) \leq \lceil \frac{n}{2} \rceil.
\end{cases}\]
\end{theorem}
\begin{proof}Let \( \rho \in \mathcal{DORP}_{n} \) be expressed as in \eqref{1}. If \( p = h(\rho) \in \{0,1, \lceil \frac{n}{2} \rceil + 1, \ldots, n\} \). Thus, by the contrapositive of Lemma \ref{rev}, we see that  \( \rho \) is not reversible; that is to say,

\[\sigma=
\begin{pmatrix}
A_1 & \dots & A_p \\
a_p & \dots & a_1
\end{pmatrix}
\notin \mathcal{DORP}_{n}.
\]
\noindent This means that \( \text{\bf Ker} \, \rho \) can only admit the image set \( \{a_{1},\ldots, a_{p}\} \) in one way. Thus, it follows from Theorem \ref{a1}(c) that \( \rho \mathcal{H}^{*} \sigma \) if and only if \( \rho = \sigma \), and consequently, \( |H^{*}_{\rho}| = 1 \).

On the other hand, if \(2 \leq h(\rho) \leq \lceil \frac{n}{2} \rceil\), then by Lemma \ref{rev}, \( \rho \) is reversible, and thus the map

\[\sigma=
\begin{pmatrix}
A_1 & \dots & A_p \\
a_p & \dots & a_1
\end{pmatrix}
\in \mathcal{DORP}_{n}.
\]
\noindent This means that \( \text{\bf Ker} \, \rho \) can only admit the image set \( \{a_{1},\ldots, a_{p}\} \) in two ways. It follows from Theorem \ref{a1}(c) that \( \rho \mathcal{H}^{*} \sigma \). Hence, \( |H^{*}_{\rho}| = 2 \), as required.
\end{proof}

\begin{lemma}\label{uaaaa} On the monoid  $\mathcal{DORP}_{n}$  \emph{(}$n\geq 4$\emph{)}, we have $\mathcal{D}^{*}=\mathcal{L}^{*}\circ\mathcal{R}^{*}\circ\mathcal{L}^{*}=\mathcal{R}^{*}\circ\mathcal{L}^{*}\circ\mathcal{R}^{*}$.
\end{lemma}
\begin{proof} The proof of $\mathcal{R}^{*}\circ\mathcal{L}^{*}\circ\mathcal{R}^{*}\subseteq \mathcal{L}^{*}\circ\mathcal{R}^{*}\circ\mathcal{L}^{*}$  follows from the converse of the  proof of (d) in the above theorem, while for $\mathcal{R}^{*}\circ\mathcal{L}^{*}\circ\mathcal{R}^{*}\supseteq \mathcal{L}^{*}\circ\mathcal{R}^{*}\circ\mathcal{L}^{*}$, we need to prove that   $\mathcal{L}^{*}\circ\mathcal{R}^{*}\neq \mathcal{R}^{*}\circ\mathcal{L}^{*}$. Take \[\rho=\left(\begin{array}{cc}
                                                                            1  &  2 \\
                                                                            {1} &2
                                                                          \end{array}
\right) \text{ and } \sigma=\left(\begin{array}{cc}
                                                                            2  &  3 \\
                                                                            {2} &3
                                                                          \end{array}
\right),\]

\noindent and define $\gamma=\left(\begin{array}{cc}
                                                                            2  &  3 \\
                                                                            {1} &2
                                                                          \end{array}
\right).$ Observe that, $\im \rho = \im \gamma$ and $\dom \gamma = \dom \sigma$, as such $\rho \mathcal{L}^{*} \gamma \mathcal{R}^{*} \sigma$. That is, $(\rho, \sigma) \in \mathcal{L}^{*} \circ \mathcal{R}^{*}$.

On the flip side, if $(\rho, \sigma)$ belongs to $\mathcal{R}^{*} \circ \mathcal{L}^{*}$, Therefore, it follows that there must be a $\delta \in \mathcal{DORP}_{n}$ such that $\rho \mathcal{R}^{*} \delta \mathcal{L}^{*} \sigma$. This leads to the conditions $\dom \rho = \dom \delta = \{1, 2\}$ and $\im \delta = \im \sigma = \{2, 3\}$, which is a contradiction. Hence, the assertion is established.
\end{proof}

\begin{lemma}\label{uaaa} On the semigroups  ${RQ}_{n}(p)$ and  $I(n, \,p)$ \emph{(}$1\leq p\leq n-1$\emph{)}, we have $\mathcal{D}^{*}=\mathcal{L}^{*}\circ\mathcal{R}^{*}\circ\mathcal{L}^{*}=\mathcal{R}^{*}\circ\mathcal{L}^{*}\circ\mathcal{R}^{*}.$
\end{lemma}
\begin{proof} The argument is identical to that used in the previous lemma.
\end{proof}

As described in \cite{FOUN2}, to define the relation $\mathcal{J}^{*}$ on a semigroup $S$, we first represent the $\mathcal{L}^{*}$-class containing an element $a \in S$ by $L^{*}_{a}$. Similar notation applies to the classes of other relations. A \emph{left} (respectively, \emph{right}) $*$-\emph{ideal} of a semigroup $S$ is defined as a \emph{left} (respectively, \emph{right}) ideal $I$ of $S$ such that $L^{*}_{a} \subseteq I$ (respectively, $R^{*}_{a} \subseteq I$) for all $a \in I$. A subset $I$ of $S$ is called a $*$-ideal of $S$ if it is both a left and a right $*$-ideal. The \emph{principal $*$-ideal} generated by an element $a \in S$, denoted by $J^{*}(a)$, is defined as the intersection of all $*$-ideals of $S$ containing $a$. The relation $\mathcal{J}^{*}$ is then defined as: $a \mathcal{J}^{*} b$ if and only if $J^{*}(a) = J^{*}(b)$.

The next lemma is essential for our continued investigation into the properties of $\mathcal{J}^{*}$ in $\mathcal{DORP}_{n}$.

\begin{lemma}[\cite{FOUN2}, Lemma 1.7]\label{jj}  Let $a$  be an element of a semigroup $S$. Then $b \in J^{*}(a)$ if and only if there are elements $a_{0},a_{1},\dots, a_{n}\in  S$, $x_{1},\dots,x_{n}, y_{1}, \dots,y_{n} \in S^{1}$ such that $a = a_{0}$, $b = a_{n}$, and $(a_{i}, x_{i}a_{i-1}y_{i}) \in \mathcal{D}^{*}$ for $i = 1,\dots,n.$
\end{lemma}

Following the work of Umar \cite{ua}, we now  have the subsequent results.
\begin{lemma}\label{jjj} For $\rho, \, \sigma \in \mathcal{DORP}_{n}$, let $\rho\in J^{*}(\sigma)$. Then $\mid \im \, \rho \mid\leq \mid \im \,\sigma \mid$.
\end{lemma}
\begin{proof}
Let $ \rho \in J^{*}(\sigma)$. Thus, by Lemma \ref{jj}, there exist $\sigma_{0}, \sigma_{1}, \dots, \sigma_{n} \in \mathcal{DORP}_{n}$, $\gamma_{1}, \dots, \gamma_{n}$, and $\tau_{1}, \dots, \tau_{n}$ in $\mathcal{DORP}_{n}$ such that $\sigma = \sigma_{0}$, $\rho = \sigma_{n}$, and $(\sigma_{i}, \gamma_{i} \sigma_{i-1} \tau_{i}) \in \mathcal{D}^{*}$ for $i = 1, \dots, n$. By Lemma \ref{uaaaa}, this implies that
\[
|\im \, \sigma_{i}| = |\im \, \gamma_{i} \sigma_{i-1} \tau_{i}| \leq |\im \, \sigma_{i-1}|,
\]
so that
\[
|\im \, \rho| \leq |\im \, \sigma|.
\]
The result is now clear.
\end{proof}

\begin{lemma}\label{uaaaaa} On the  monoid   $\mathcal{DORP}_{n}$, we have $\mathcal{D}^{*}=\mathcal{J}^{*}$.
\end{lemma}
\begin{proof}To prove the result, it is sufficient to show that $\mathcal{J}^{*} \subseteq \mathcal{D}^{*}$ (since we already know that $\mathcal{D}^{*} \subseteq \mathcal{J}^{*}$). Let us assume $(\rho, \sigma) \in \mathcal{J}^{*}$. Then, by definition, $J^{*}(\rho) = J^{*}(\sigma)$, which implies that $\rho \in J^{*}(\sigma)$ and $\sigma \in J^{*}(\rho)$. Using Lemma \ref{jjj}, it follows that
\[
|\im \, \rho| = |\im \, \sigma|.
\]
Therefore, by Theorem \ref{a1} (d), we see that
\[
\mathcal{J}^{*} \subseteq \mathcal{D}^{*},
\]
as required.
\end{proof}

\begin{lemma}\label{un} On the semigroup $S$ in $\{\mathcal{DORP}_{n},  \,  I(n, \,p), \, RQ_{p}(n) \}$, every $\mathcal{R}^{*}$-class  includes a single unique idempotent.
\end{lemma}
\begin{proof} This result follows from the fact that the semigroup $S$ is abundant and the order decreasing property of elements in $S$.
\end{proof}

\begin{remark}\label{remm}\begin{itemize}
             \item[(a)] It is apparent that, for each $p$ such that $1 \leq p \leq n$, the number of $\mathcal{R}^{*}$-classes contained in $J^{*}_{p} = \{\rho\in \mathcal{DORP}_{n}: \, |\im \, \rho|=p\}$ coincides with the total number of partially ordered partitions of $[n]$ into $p$ parts. This number coincide with the number of $\mathcal{R}$-classes in $\{\rho\in \mathcal{OP}_n: \, |\im \, \rho|=p\}$, i.e., $\sum\limits_{r=p}^{n}{\binom{n}{r}}{\binom{r-1}{p-1}}$, as shown in \emph{[\cite{al3}, Lemma 4.1]}.
             \item[(b)] If $S \in \{{RQ}_{p}(n),  \,  I(n, \,p)\}$, then the characterizations of the starred Green's relations presented in Theorem \ref{a1} are valid in $S$, as  such the conditions stated in (i) also hold in $S$. Moreover, for any $S\in \{{RQ}_{p}(n),  \,  I(n, \,p) \}$, $S$ is an inverse ideal of $\mathcal{P}_{n}$.
           \end{itemize}
\end{remark}
Thus, following Remark \ref{remm}(ii), we present the following result.
\begin{lemma} If $S\in \{{RQ}_{p}(n),  \,  I(n, \,p) \}$, then $S$ is abundant.
\end{lemma}

A semigroup $S$  with \textbf{0} is called $0-^*$\emph{bisimple} if it has a unique nonzero $\mathcal{D}^{*}-$class \cite{umar}. Thus, we have now shown the following result.
\begin{theorem} Let $RQ_{n}(p)$ be as defined in \eqref{knn}. Then $RQ_{n}(p)$ is a non-regular $0-^*$bisimple abundant semigroup.
\end{theorem}

Notice that the semigroup $I(n, \,p)$, much like $\mathcal{DORP}_{n}$, can be expressed as the union of $\mathcal{J}^{*}$ classes:
\[
J_{0}^{*}, \, J_{1}^{*}, \, \dots, \, J_{p}^{*},
\]
where
\[
J_{p}^{*} = \{\rho \in I(n, \,p) : |\im \, \rho| = p\}.
\]

Furthermore, $I(n, \,p)$ contains $\sum_{r=p}^{n} \binom{n}{r} \binom{r-1}{p-1}$ $\mathcal{R}^{*}$-classes and $\binom{n}{p}$ $\mathcal{L}^{*}$-classes in each $J^{*}_{p}$. As a result, the Rees quotient semigroup ${RQ}_{p}(n)$ has
\(
\sum_{r=p}^{n} \binom{n}{r} \binom{r-1}{p-1} + 1
\)
$\mathcal{R}^{*}$-classes and
\(
\binom{n}{p} + 1
\)
$\mathcal{L}^{*}$-classes. (The additional term 1 corresponds to the singleton class containing the zero element in each case.)

Prior to presenting the subsequent result, we note that the number of idempotents in the large Schr\"{o}der monoid $\mathcal{LS}_{n}$ is expressed as $\frac{3^{n}+1}{2}$ (see \cite{al2}, Proposition 3.5). We will now state the theorem below.

  \begin{theorem} Let $\mathcal{DORP}_{n}$ be as defined in \eqref{qn2}. Then  $|E(\mathcal{DORP}_{n})|=\frac{3^{n}+1}{2}$.
  \end{theorem}
  \begin{proof} This is due to the fact that the number of $\mathcal{R}^{*}$-classes in $\mathcal{LS}_{n}$ is the same as that in $\mathcal{DORP}_{n}$. \end{proof}

\section{Rank properties}
 For a  semigroup $S$  and a   nonempty subset of $S$, say $T$, the  \emph{smallest subsemigroup} of $S$ that contains $T$ is denoted by $\langle T \rangle$ and is referred to as the\emph{subsemigroup generated by $T$}. If  $T$  is a finite subset of  $S$ such that $\langle T \rangle$ equals $S$, then $S$ is said to be  a \emph{finitely-generated semigroup}. The \emph{rank} of a finitely generated semigroup $S$ is defined and denoted as:
\[
\text{rank}(S) = \min\{|T| : \langle T \rangle = S\}.
\]
\noindent If the set $A$ consists only of idempotents in $S$, then $S$ is said to be an \emph{idempotent-generated semigroup} (equivalently, a \emph{semiband}), and the idempotent-rank is denoted by $\text{idrank}(S)$. For a more comprehensive discussion on the rank properties of a semigroup, we direct the reader to \cite{hrb, hrb2}. There are several categories of transformation semigroups whose ranks have been studied; see, for instance, \cite{ gu1, gm, gm3, hf, hrb, hrb2, HRS, zm1}. The work in \cite{dm} examines the rank of the Schr{\"o}der monoid $\mathcal{LS}_{n}$ (by \emph{anti isomorphism}), whereas the authors in \cite{cn} (by \emph{anti isomorphism}) computed the ranks of its certain two-sided ideals and their corresponding Rees quotients. In fact, it has been demonstrated that this semigroup  is idempotent-generated in \cite{gmv}. Our objective is to determine the rank of the two-sided ideal $I(n, \,p)$ of the monoid $\mathcal{DORP}_{n}$, which will consequently provide the rank of the monoid $\mathcal{DORP}_{n}$ as a specific case.

We commence our study of the rank properties of the semigroup $S \in \{I(n, \,p), \mathcal{DORP}_{n}\}$ by first presenting the following definition related to injective  antitone elements of height $2 \leq p \leq \lceil \frac{n}{2} \rceil$.

\begin{definition} An injective antitone  map in $\mathcal{DORP}_{n}$,  say $\delta$,  of height $2 \leq p \leq \lceil \frac{n}{2} \rceil$,   is said  to be a \emph{vital element} if it is of the form: \begin{equation}\label{requi}\delta=\begin{pmatrix}y_{p}&y_{p}+1&\cdots& y_{p}+p-1\\y_{p}&y_{p-1}&\cdots& y_{1}\end{pmatrix},\end{equation} where $1\le y_{1}<\cdots<y_{p}<y_{p}+1<\cdots<y_{p}+p-1\leq n$.
\end{definition}

\begin{remark} \label{requ}Observe that the domain of a vital element \emph{($\dom \, \rho$)} is convex, and the minimum element in  $\dom \, \rho$ is a fixed point.
\end{remark}

 We immediately present the following lemma.
\begin{lemma}\label{ms} For $2 \leq p \leq \lceil \frac{n}{2} \rceil$, every  $\mathcal{L}^{*}-$class of an element, say   $\rho$, of height $p$ in the monoid $\mathcal{DORP}_{n}$ contains a  unique vital element.
\end{lemma}
 \begin{proof}  Let $\rho \in \mathcal{DORP}_{n}$ with $h(\rho)=p$, where $ 2 \leq p \leq \lceil \frac{n}{2} \rceil$. Then $\rho$ is either an isotone map or an antitone map. In the former, let $\rho$ be as  expressed as in \eqref{1}; and in the latter, let $\rho$ be as expressed in \eqref{2}.  Consider $L^{*}_{\rho}$, now it is clear that in either  case, $a_{p}=\max  (\im \, \rho)$. Moreover, $ a_{1}<\cdots <a_{p}<a_{p}+1<\cdots <a_{p}+p-1$. Thus, the map
 \[\delta=\begin{pmatrix}a_{p}&a_{p}+1& \cdots& a_{p}+p-1\\a_{p}&a_{p-1}&\cdots&a_{1}\end{pmatrix}\]
\noindent is an antitone map in  $\mathcal{DORP}_{n}$, where $\im \, \rho= \im \, \delta$, and so $\delta\in L^{*}_{\rho}$. Clearly $\dom \, \delta$ is convex, and   $a_{p}=\min ( \dom \, \delta)$ is a fixed point of $\delta$. Therefore, $\delta$ is the unique vital element in $L^{*}_{\rho}$.
\end{proof}

 The next lemma is concerned with  the factorization of injective antitone maps  in the monoid $\mathcal{DORP}_{n}$ with  convex domains.
\begin{lemma}\label{inj} Every injective antitone element  $\rho\in\mathcal{DORP}_{n}$ with a convex domain is a product of idempotent(s) and the unique vital element that belongs to  ${L}^{*}_{\rho}$.
\end{lemma}

\begin{proof} Consider an injective  antitone map of height $p$ with a convex domain, say $\rho$, in $\mathcal{DORP}_{n}$ of the form:
\[\rho =\begin{pmatrix}t&t+1& \cdots& t+p-1\\a_{p}&a_{p-1}&\cdots&a_{1}\end{pmatrix},\]
\noindent  where $2\leq p\leq \lceil \frac{n}{2} \rceil$ and $1\leq a_{1}<\cdots<a_{p}\leq t<\cdots<t+p-1\leq n$  for some $t\in \{2, \ldots, n-p+1\}$.

Now, for $1\leq i\leq p$, define \[\varepsilon_{i}=\begin{pmatrix}a_{p}&a_{p}+1&\cdots&a_{p}+i-2&\{a_{p}+i-1, t+i-1\}&t+i&\cdots& t+p-1\\a_{p}&a_{p}+1&\cdots&a_{p}+i-2&a_{p}+i-1&t+i&\cdots&t+p-1\end{pmatrix},\]

\[\epsilon=\begin{pmatrix}t&t+1& \cdots& t+p-1\\t&t+1&\cdots&t+p-1\end{pmatrix} \text{ and }\delta=\begin{pmatrix}a_{p}&a_{p}+1& \cdots& a_{p}+p-1\\a_{p}&a_{p-1}&\cdots&a_{1}\end{pmatrix}.\]

\noindent Notice that $a_{p} \leq t$, which implies that $a_{p} + i - 1 \leq t + i - 1$ for all $1 \leq i \leq p$, and so $\varepsilon_{i}$ is an isotone idempotent of height $p$ in $\mathcal{DORP}_{n}$. Moreover, it is clear that $a_{1} < \cdots < a_{p} < a_{p} + 1 < \cdots < a_{p} + p - 1$. Thus, $\delta$ is an antitone injective map with a convex domain that fixes $a_{p} = \min ( \dom \, \delta)$. Hence, $\delta$ is a vital element. Furthermore, $\im \, \rho = \im \, \delta$, which implies that $\delta \in L^{*}_{\rho}$. It is also  clear that $\epsilon$ is an idempotent in $\mathcal{DORP}_{n}$ of height $p$.

\noindent Now it is not difficult to see that:

\noindent\resizebox{1.01\textwidth}{!}{%
\begin{minipage}{\textwidth}
\begin{align*}\epsilon\varepsilon_{1}\varepsilon_{2}\cdots\varepsilon_{p}\delta=&
\begin{pmatrix}
t & t+1&  \cdots &  t+p-1 \\
t &t+1&  \cdots &  t+p-1
\end{pmatrix}
\begin{pmatrix}
\{a_{p},t\} & t+1 & \cdots &  t+p-1 \\
a_p & t+1 & \cdots &  t+p-1
\end{pmatrix}
\begin{pmatrix}
a_p & \{a_{p}+1,t+1\} & t+{2} & \cdots &  t+p-1 \\
a_p & a_{p}+1 & t+{2} & \cdots &  t+p-1
\end{pmatrix}\\&
\cdots
\begin{pmatrix}
a_p &a_{p}+1& \cdots & a_{p}+p-2 & \{a_{p}+p-1, t+p-1\} \\
a_p & a_{p}+1& \cdots & a_{p}+p-2 & a_{p}+p-1
\end{pmatrix}\begin{pmatrix}
a_p&a_{p}+1 & \cdots &  a_{p}+p-1 \\
a_p & a_{p-1}& \cdots &  a_1
\end{pmatrix}
\\=& \begin{pmatrix}
t &t+1& \cdots &  t+p-1 \\
a_p&a_{p-1} &  \cdots &  a_1
\end{pmatrix}
= \rho,
\end{align*}
\end{minipage}}
\noindent as required. The result now follows.
\end{proof}

We now present the following theorem.

\begin{theorem}\label{hq} The monoid $\mathcal{DORP}_{n}$ is generated by idempotent(s) and vital elements.
\end{theorem}
\begin{proof} Let $\rho\in \mathcal{DORP}_{n}$. Then, $\rho$ is either an isotone map or an antitone map.

\noindent\textbf{(i.)} If $\rho$ is an isotone map, then  $\rho \in \mathcal{LS}_{n}$. Thus, by [\cite{gmv}, Theorem 14.4.5], $\rho$ is idempotent-generated.

 \noindent\textbf{(ii.)} Now suppose  $\rho$ is an antitone map of the form \[\rho=\begin{pmatrix}A_{1}&\cdots& A_{p}\\a_p&\cdots&a_{1}\end{pmatrix}\] \noindent with $2\le p \le \lceil \frac{n}{2} \rceil$. Now let $t_{i}=\min A_{i}$ for all $1\le i\le p$ and suppose $s_{i+1}=t_{i+1}-t_{i}$ for all $1\le i\le p-1$. Next,  let $T_{i+1}=\{t_{1}+1, t_{1}+2, \ldots, t_{1}+s_{2}+s_{3}+\cdots+s_{i+1}\}$ be a convex set. Observe that:
\begin{align*}
t_{1}+s_{2}+s_{3}+\cdots+s_{i+1}&=t_{1}+(t_{2}-t_{1})+(t_{3}-t_{2})+\cdots+\\&(t_{i}-t_{i-1})+(t_{i+1}-t_{i})\\&=t_{i+1}=\min A_{i+1}.
\end{align*}
 This shows that for each $1\leq i\leq p-1$, $t_{i+1}=\max T_{i+1}$. Thus, \[T_{i+1}=\{t_{1}+1, \, t_{1}+2, \, \ldots, \, t_{1}+s_{2}+s_{3}+\cdots+s_{i+1}\}=\{t_{1}+1, \,  t_{1}+2, \,  \ldots, \, t_{i+1}\}.\] \noindent So, the maps \[\epsilon=\begin{pmatrix}A_{1}&\cdots& A_{p}\\t_1&\cdots&t_{p}\end{pmatrix}\]
\noindent and \[\xi_{i}=\begin{pmatrix}t_{1}&\cdots& t_{1}+i-1&\{t_{1}+i,\ldots, t_{i+1}\}&t_{i+2}&\cdots& t_{p}\\t_{1}&\cdots&t_{1}+i-1&t_{1}+i&t_{i+2}&\cdots&t_{p}\end{pmatrix},\] are  idempotents of height $p$ in $\mathcal{DORP}_{n}$ for all $1\leq i\leq p$.

\noindent Notice also that since $a_{p}\leq t_{1}$, it follows that  $a_{p}\leq t_{1}<t_{1}+1<\cdots<t_{1}+p-1$. Therefore, the map $\sigma$ defined as \[\sigma=\begin{pmatrix}t_{1}&t_{1}+1&\cdots& t_{1}+p-1\\a_p&a_{p-1}&\cdots&a_{1}\end{pmatrix}\]
\noindent is an antitone map with a convex domain, and so by Lemma \ref{inj}, $\sigma$ is generated by idempotent(s) and the unique vital element in $L^{*}_{\sigma}$.

 Now observe that

\noindent\resizebox{1.01\textwidth}{!}{%
\begin{minipage}{\textwidth}
\begin{align*}\epsilon \xi_{1}\cdots \xi_{p}\sigma=&
\begin{pmatrix}A_{1}&\cdots& A_{p}\\
t_1&\cdots&t_{p}\end{pmatrix}
\begin{pmatrix}t_{1}&\{t_{1}+1, t_{1}+2, \ldots, t_{2}\}& t_{3}&\cdots& t_{p}\\
t_1&t_{1}+1&t_{3}&\cdots&t_{p}\end{pmatrix}
\begin{pmatrix}
t_1 & t_{1}+1 & \{t_{1}+2,t_{1}+3, \ldots, t_{3}\}&t_{4} & \cdots &  t_p \\
t_1 & t_{1}+1 & t_{1}+{2}&t_{4} & \cdots &  t_p
\end{pmatrix}\\&
\cdots
\begin{pmatrix}
t_1 &t_{1}+1& \cdots & t_{1}+p-2 & \{t_{1}+p-1,\ldots, t_{p}\} \\
t_1 & t_{1}+1& \cdots & t_{1}+p-2 & t_{1}+p-1
\end{pmatrix}\begin{pmatrix}
t_1&t_{1}+1 & \cdots &  t_{1}+p-1 \\
a_p & a_{p-1}& \cdots &  a_1
\end{pmatrix}
\\=& \begin{pmatrix}
A_{1} &A_{2}& \cdots &  A_p \\
a_p&a_{p-1} &  \cdots &  a_1
\end{pmatrix}
= \rho,
\end{align*}
\end{minipage}}
\noindent as required. Notice that $\im \, \rho =\im \, \sigma$,  and so   $L^{*}_{\sigma}=L^{*}_{\rho}$.
Hence, $\rho$ is generated by idempotents and the unique vital element in $L^{*}_{\rho}$.   The result now follows.
 \end{proof}
Now we introduce the following definition to continue with our investigation.
 \begin{definition} A vital element $\delta_{p, \, i}$  of height $2 \leq p \leq \lceil \frac{n}{2} \rceil$ is called \emph{convex} if it is of the form   \begin{equation}\label{opt} \delta_{p, \, i}=\begin{pmatrix}i&i+1&\cdots& i+p-1\\i&i-1&\cdots&i-p+1\end{pmatrix},\end{equation} where $p\leq i\leq n-p+1$.
\end{definition}

We will now outline the next lemma.

 \begin{lemma}\label{inj2} Every non-convex vital element is a product of a convex vital element and idempotent element(s).
\end{lemma}
\begin{proof} Let $\delta$ be a non-convex vital element of height $2 \leq p \leq \lceil \frac{n}{2} \rceil$, as expressed in \eqref{requi}. This means that there exists $0\leq i\le p-2$ such that $\{y_{p-i}, y_{p-i+1}, \ldots, y_{p}\}$ is convex and $y_{p-(i+1)}-y_{p-i}>1$. That is to say, \[\{y_{p-i}, \, y_{p-i+1}, \,  \ldots, \, y_{p-1}, \,  y_{p}\}=\{(y_{p-i}=)y_{p}-i, \,  y_{p}-i+1, \,  \ldots, \, y_{p}-1,  \, y_{p}\}.\] As such $\delta$ is of the form

\[
\delta = \left(
\begin{array}{ccccc|ccc}
y_{p} & y_{p}+1 & \cdots & y_{p}+i-1 & y_{p}+i & y_{p}+i+1 & \cdots & y_{p}+p-1 \\
y_{p} & y_{p}-1 & \cdots & y_{p}-i+1 & y_{p}-i & y_{p-(i+1)} & \cdots & y_{1}
\end{array}
\right).
\]
This means that the subset (of $\im \, \delta$) $\{y_{1}, \dots, y_{p-(i+1)}\}$ needs not be  convex. Notice that \[\{y_{p-i}-1, \,  y_{p-i}-2, \, \ldots, \,  y_{p-i}-(p-i-2), \, y_{p-i}-(p-i-1)\}\] \noindent is a convex set consisting of translates  of $y_{p-i}$, which can be replaced with  $y_{p}-i$ (since $y_{p-i}=y_{p}-i$) as follows:
\[\{y_{p}-i-1,  \,   y_{p}-i-2, \,  \ldots,  \,  y_{p}-p+2, \,  \,  y_{p}-p+1\}.\]

\noindent Thus, the set $\{y_{p} - p + 1, \,  y_{p} - p + 2, \,  \ldots, y_{p} - i - 1, \,  y_{p} - i, \,  y_{p} - i + 1, \,  \ldots,  \, y_{p} - 1, \,  y_{p}\}$ is convex. Therefore, the map $\delta^{*}$ defined as
$$
\delta^{*}_{p, \, i} = \left(
\begin{array}{ccccc|cccc}
y_{p} & y_{p} + 1 & \cdots & y_{p} + i - 1 & y_{p} + i & y_{p} + i + 1 & \cdots & y_{p} + p - 1 \\
y_{p} & y_{p} - 1 & \cdots & y_{p} - i + 1 & y_{p} - i & y_{p} - i - 1& \cdots & y_{p} - p + 1
\end{array}
\right)
$$

\noindent is a convex vital element. Additionally, the map $\gamma$ defined as

$$
\gamma = \begin{pmatrix} y_{p} - p + 1 & y_{p} - p + 2 & \cdots & y_{p} - i - 1 & y_{p} - i & y_{p} - i + 1 & \cdots & y_{p} - 1 & y_{p} \\ y_{1} & y_{2} & \cdots & y_{p-i-1} & y_{p-i} & y_{p-i+1} & \cdots & y_{p-1} & y_{p} \end{pmatrix}
$$

\noindent is an injective decreasing isotone map that is idempotent-generated by [\cite{gmv}, Theorem 14.4.5]. Now, clearly, $\delta^{*}_{p, \, i}\gamma = \delta$. The result now follows.
\end{proof}

 Now, for $2 \leq p \leq \lceil \frac{n}{2} \rceil$, let $M(p)$ be the set of all convex vital elements in ${RQ}_{n}(p)$. Then, We now present the following lemma.

  \begin{lemma}\label{nid}   \begin{itemize} \item[(i)]For  $2 \leq p \leq \lceil \frac{n}{2} \rceil$,  $|M(p)|=n-2p+2$; \item[(ii)] For $1\le p\leq n-1$, we have $|E({RQ}_{n}(p)\setminus\{\textbf{0}\})|=\sum\limits_{r=p}^{n}{\binom{n}{r}}{\binom{r-1}{p-1}}$.\end{itemize}
\end{lemma}
\begin{proof}
\noindent \textbf{(i.)} Notice that the domain of each convex vital element of height $2 \leq p\leq \lceil \frac{n}{2} \rceil$  has its minimum element, say $i$, within   the closed interval $p\leq i\leq n-p+1$. The result follows easily by counting the elements within this range.

\noindent\textbf{(ii.)} The result follows from Remark \ref{remm}.
\end{proof}
Next, we have the following theorem.
\begin{theorem} Let $\mathcal{DORP}_{n}$ be as defined in \eqref{qn1}. Then the number of convex vital elements of height $2 \leq p\leq \lceil \frac{n}{2} \rceil$ in $\mathcal{DORP}_{n}$ is \[(n+2)\left(\left\lceil \frac{n}{2} \right\rceil-1\right)-\left(\left\lceil \frac{n}{2} \right\rceil\left(\left\lceil \frac{n}{2} \right\rceil+1\right)-2\right).\]
\end{theorem}
\begin{proof} By summing the convex vital elements in Lemma \ref{nid}(i) over $2 \leq p \leq \lceil \frac{n}{2} \rceil,$ together with some algebraic manipulations, we obtain \[\sum_{p=2}^{\left\lceil \frac{n}{2} \right\rceil} (n - 2p + 2) = (n + 2)\left(\left\lceil \frac{n}{2} \right\rceil - 1\right) - \left(\left\lceil \frac{n}{2} \right\rceil\left(\left\lceil \frac{n}{2} \right\rceil + 1\right) - 2\right).\]
\end{proof}

Now, let
\[G(p) =
\begin{cases}
E(RQ_{n}(p) \setminus \{\textbf{0}\}), & \text{if }  p \in \{1, \lceil \frac{n}{2} \rceil + 1, \ldots, n\}; \\
M(p) \cup E(RQ_{n}(p) \setminus \{\textbf{0}\}), & \text{if }2 \leq p \leq \lceil \frac{n}{2} \rceil.
\end{cases}\]

\noindent The following result demonstrates that the set $G(p)$ serves as the minimal generating set for ${RQ}_{n}(p)$.
\begin{lemma}\label{minnn}
Let $\rho$ and $\sigma$ be elements of ${RQ}_{n}(p)$. Then $\rho\sigma \in G(p)$ if and only if both $\rho$ and $\sigma$ belong to $G(p)$, and either $\rho\sigma = \rho$ or $\rho\sigma = \sigma$.
\end{lemma}

\begin{proof} Suppose $\rho\sigma\in G(p)$.\\

\noindent  If $ p \in \{1, \lceil \frac{n}{2} \rceil + 1, \ldots, n\}$, then $\rho\sigma\in E(RQ_{n}(p) \setminus \{\textbf{0}\})$. Consequently, the result is obtained from [\cite{gmv}, Theorem 14.4.5].\\

\noindent  Now if $2 \leq p \leq \lceil \frac{n}{2} \rceil$, then  either $\rho\sigma \in M(p)$ or $\rho\sigma \in E(RQ_{n}(p) \setminus \{\textbf{0}\})$.
If  $\rho\sigma\in  E(RQ_{n}(p) \setminus \{\textbf{0}\})$, then the result follows from [\cite{gmv}, Theorem 14.4.5] as well.\\

 \noindent Now suppose  $\rho\sigma\in M(p)$. Thus, $\rho\sigma$ is a convex vital element, and it has the form \[\rho\sigma=\begin{pmatrix}i&i+1& \cdots &i+p-1\\i&i-1&\cdots&i-p+1\end{pmatrix},\] \noindent where $p\leq i\leq n-p+1$.
This signifies that $\dom \, \rho=\dom \, \rho\sigma$, $\im \, \sigma =\im \, \rho\sigma$ and  $\im \, \rho=\dom \, \sigma$. Thus, \[\rho=\begin{pmatrix}i& {i+1}&\cdots& i+p-1\\i\rho &(i+1)\rho&\cdots& (i+p-1)\rho\end{pmatrix} \text{ and } \sigma=\begin{pmatrix}i\rho &(i+1)\rho&\cdots&(i+p-1)\rho\\i&i-1&\cdots&i-p+1\end{pmatrix}.\] \\

 \noindent The assertion is that \(\rho\) has to be an idempotent. Notice that $\rho$ and $\sigma$ are decreasing maps. Thus, $i\rho\leq i$ and $i=(i\rho)\sigma\leq i\rho$. This ensures that $i\rho=i$. Moreover, for any $i-1\leq j\leq i-p+1$ we see that $j\leq (j+2)\rho\leq j+2$. This means that for all $i-1\leq j\leq i-p+1$, either $(j+2)\rho=j$; or $(j+2)\rho=j+1$; or $(j+2)\rho=j+2$.\\

\noindent If $(j + 2)\rho = j$, then in particular, for $j = i - 2$, we see that  $(i - 2 + 2)\rho = i - 2$, that is, $i\rho = i - 2$, which contradicts  the fact that $i\rho = i$.\\

\noindent  Now, if $(j + 2)\rho = j + 1$ for all $i-1\leq j\leq i-p+1$, then in particular, for $j = i - 1$, we have $(i - 1 + 2)\rho = i - 1 + 1$, that is, $(i + 1)\rho = i$, which also contradicts the fact that $i\rho = i$. Hence, we conclude that $(j + 2)\rho = j + 2$ for all $i-1\leq j\leq i-p+1$. This ensures that:

$$
\rho = \begin{pmatrix} i & i + 1 & \cdots & i + p - 1 \\ i & i + 1 & \cdots & i + p - 1 \end{pmatrix}
\text{ and }
\sigma = \begin{pmatrix} i & i + 1 & \cdots & i + p - 1 \\ i & i - 1 & \cdots & i - p + 1 \end{pmatrix}.
$$

\noindent Therefore, $\sigma \in M(p) \subset G(p)$ and $\rho \in E({RQ}_{n}(p) \setminus \{\textbf{0}\}) \subset G(p)$, and also $\rho \sigma = \sigma.$

The converse is obvious.
\end{proof}

At this point, we have established a significant finding in this section.

\begin{theorem}\label{pb} Let ${RQ}_{n}(p)$ be as defined in \eqref{knn}. Then  \[\text{rank } {RQ}_{n}(p) =\left\{
                                                                                                                    \begin{array}{ll}

\sum\limits_{r=p}^{n}{\binom{n}{r}}{\binom{r-1}{p-1}}, & \hbox{ if  $p \in \{1, \lceil \frac{n}{2} \rceil + 1, \ldots, n\}$;} \\
                                                                                                                      (n-2p+2)+\sum\limits_{r=p}^{n}{\binom{n}{r}}{\binom{r-1}{p-1}}, & \hbox{if  $2 \leq p \leq \lceil \frac{n}{2} \rceil$.}
                                                                                                             \end{array}
                                                                                                                  \right.\]
\end{theorem}

\begin{proof} The proof follows from Lemmas \ref{nid} and \ref{minnn}.
\end{proof}

 The next lemma plays a vital role in establishing the ranks of the monoid $\mathcal{DORP}_{n}$ and its two-sided ideal $I(n, \,p)$.  Now,  for $0\leq p\leq n-1$ let \[J^{*}_{p}=\{\rho\in \mathcal{DORP}_{n}: \,  |\im \, \rho|=p \}.\] Moreover, for $2\leq p\leq \lceil\frac{n}{2}  \rceil$ let $M(p)$ be the collection of all convex vital elements in $J_{p}^{*}$, and let \begin{equation}\label{gp}G(p) =
\begin{cases}
E(J^{*}_{p}), & \text{if } p \in \{0, 1, \lceil \frac{n}{2} \rceil + 1, \ldots, n\}; \\
M(p) \cup E(J^{*}_{p}), & \text{if } 2 \leq p \leq \lceil \frac{n}{2} \rceil.
\end{cases}\end{equation}
 \noindent We can then state the following lemma.

  \begin{lemma}\label{nid2} Let $G(p)$ be as defined in \eqref{gp}. Then \[|G(p)| =\left\{\begin{array}{lll}
1, & \hbox{ if  $p=0$;} \\
\sum\limits_{r=p}^{n}{\binom{n}{r}}{\binom{r-1}{p-1}}, & \hbox{ if  $p \in \{1, \lceil \frac{n}{2} \rceil + 1, \ldots, n\}$;} \\
(n-2p+2)+\sum\limits_{r=p}^{n}{\binom{n}{r}}{\binom{r-1}{p-1}}, & \hbox{if  $2 \leq p \leq \lceil \frac{n}{2} \rceil$.} \end{array}\right.\]
\end{lemma}
\begin{proof}  If $p=0$, the result is trivial. The other both cases  follow directly from Lemma \ref{nid}.
\end{proof}

Let us recall from \eqref{opt}  that, for $2\le p \le \lceil \frac{n}{2} \rceil$ and any \( p\leq i\leq  n-p+1\) the convex vital  element of height $p$ has the form: \begin{equation*}\label{opt4}\delta_{p, \, i}=\begin{pmatrix}i& {i+1}&\cdots& i+p-1\\i &i-1&\cdots& i-p+1\end{pmatrix}.\end{equation*}
It is important to note the two extreme cases: when \( i = p \) and when \( i = n - p + 1 \). These cases are represented by

\[
\delta_{p, \, p} = \begin{pmatrix} p & p + 1 & \cdots & 2p - 1 \\ p & p - 1 & \cdots & 1 \end{pmatrix}
\text{ and }
\delta_{p, \, n - p + 1} = \begin{pmatrix} n - p + 1 & n - p + 2 & \cdots & n - 1 & n \\ n - p + 1 & n - p & \cdots & n - 2p + 1 & n - 2p + 2 \end{pmatrix},
\]
\noindent respectively. Henceforth, we shall refer to $\delta_{p, \, p}$ and $\delta_{p, \, n-p+1}$ as \emph{extreme convex vital} elements.

 Moreover, if \( n \) is odd and \( p = \left\lceil \frac{n}{2} \right\rceil \), then \( p = \lceil \frac{n}{2} \rceil = \frac{n+1}{2} \). Thus,
\[
n - p + 1 = n - \frac{n+1}{2} + 1 = \frac{n+1}{2} = \left\lceil \frac{n}{2} \right\rceil = p.
\]
Hence, if \( n \) is odd and \( p = \lceil \frac{n}{2} \rceil \), then \( \delta_{p, \,  \lceil \frac{n}{2} \rceil } = \delta_{p, \, n - \lceil \frac{n}{2} \rceil  + 1} \). This implies that there is only one extreme convex vital element at height \( p = \lceil \frac{n}{2} \rceil\) whenever \( n\) is odd, which is \(  \delta_{p, \, \lceil \frac{n}{2}\rceil}\) and $n>1$. \\

However, if \( n\) is even, and \( p = \lceil \frac{n}{2} \rceil =\frac{n}{2}\), then \[n- p+1 = n-\left\lceil \frac{n}{2} \right\rceil+1=n-\frac{n}{2}+1=\frac{n+2}{2}\neq \frac{n}{2}= \left\lceil \frac{n}{2} \right\rceil=p.\]

\noindent This demonstrates that \( \delta_{p, \, \lceil \frac{n}{2}\rceil}\neq \delta_{p, \, n-\lceil \frac{n}{2}\rceil+1}\), meaning that there are two extreme convex vital elements: \( \delta_{p, \, \frac{n}{2}}\) and \( \delta_{p, \, \frac{n+2}{2}}\), whenever $n$ is even and $n>2$.

We now present the following lemma.
\begin{lemma}\label{prod} Let $2\leq p\leq \lceil \frac{n}{2} \rceil-1$. Then, for $p+1\leq i\leq n-p$, the convex vital element $\delta_{p, \, i}$ in $J^{*}_{p}$ as expressed in \eqref{opt}, can be written as a product of a convex vital element and an idempotent element  in $J^{*}_{p+1}$.
\end{lemma}
\begin{proof} Let $\delta_{p, \, i}$ be a convex vital element in $J^{*}_{p}$ as expressed in \eqref{opt}, where $p + 1 \leq i \leq n - p$. Notice that adding and subtracting $p$ to the range $p + 1 \leq i \leq n - p$ implies $2p + 1 \leq i + p \leq n$ and $1 \leq i - p \leq n - 2p$, respectively. Thus, we see that $i + p \leq n$ and $i - p \geq 1$, which ensures that the map defined as

$$
\delta_{p, \, i}^{\prime} = \begin{pmatrix} i & {i + 1} & \cdots & i + p - 1 & i + p \\ i & i - 1 & \cdots & i - p + 1 & i - p \end{pmatrix}
$$
\noindent is a convex vital element in $J^{*}_{p+1}$. Notice that the map $\epsilon$ defined as

$$
\epsilon = \begin{pmatrix} i - p + 1 & {i - p + 2} & \cdots & i & i + 1 \\ i - p + 1 & i - p + 2 & \cdots & i & i + 1 \end{pmatrix}
$$
\noindent is in $E(J^{*}_{p+1})$.  It is now straightforward to observe that
\begin{align*}
\delta_{p, \, i}^{\prime} \epsilon &= \begin{pmatrix} i & {i + 1} & \cdots & i + p - 1 & i + p \\ i & i - 1 & \cdots & i - p + 1 & i - p \end{pmatrix}
\begin{pmatrix} i - p + 1 & {i - p + 2} & \cdots & i & i + 1 \\ i - p + 1 & i - p + 2 & \cdots & i & i + 1 \end{pmatrix} \\
&= \begin{pmatrix} i & {i + 1} & \cdots & i + p - 1 \\ i & i - 1 & \cdots & i - p + 1 \end{pmatrix} = \delta_{p,\, i},
\end{align*}
as required.
\end{proof}

\begin{remark}\label{phf}
It is crucial to point out that in the proof of the aforementioned lemma, if we allow, for example, $i = p$, then the inflation of the domain of $\delta_{p, \, i}$ to obtain $\delta_{p, \, i}^{\prime}$ by inserting $i + p \mapsto i - p$ would lead to $2p \mapsto 0$, which is impossible. Similarly, if $i = n - p + 1$, then $i + p \mapsto i - p$ will imply $n + 1 \mapsto n - 2p + 1$, which is also impossible. Thus, this ensures that there is no convex vital element and idempotent of  higher height that can generate $\delta_{p, \, p}$ and $\delta_{p, \, n - p + 1}$.
\end{remark}

We now present the following lemma.

\begin{lemma}\label{lm1} \[J^{*}_{p}\subset\left\{
                                                                                                                    \begin{array}{ll}

\langle J^{*}_{p+1}\rangle, & \hbox{ if  $p \in \{0, 1, \lceil \frac{n}{2} \rceil + 1, \ldots, n-2\}$;} \\
 \langle J^{*}_{p+1}, \delta_{p, \, p},\, \,  \delta_{p, \, n-p+1}\rangle, & \hbox{ if  $\, 2 \leq p \leq \lceil \frac{n}{2} \rceil$.}
                                                                                                             \end{array}
                                                                                                                  \right.\]

\end{lemma}
\begin{proof} \noindent\textbf{i.} If $p \in \{0, 1, \lceil \frac{n}{2} \rceil + 1, \ldots, n-2\}$, then  using Theorem \ref{hq}, It is enough to show that every element in $G(p)$ can be written as a product of elements in $G(p + 1)$. This means that every idempotent of height $p$ can be formed from the product of idempotents of height $p + 1$. Notice that idempotents of height $p$ in $\mathcal{DORP}_{n}$ are also idempotents of height $p$ in $\mathcal{LS}_{n}$. Thus, the result follows from [\cite{gmv}, Lemma 14.4.2].

\vspace{0.2cm}

\noindent\textbf{ii.} Now suppose $2 \leq p \leq \lceil \frac{n}{2} \rceil$. The idempotents in $G(p)$ as products of idempotents in $G(p+1)$ have been established by [\cite{gmv}, Lemma 14.4.2], and  so it is sufficient to show that every convex vital element in $G(p)$ can be expressed as a product of vital elements in $G(p + 1) \cup \{\delta_{p, \, p}, \, \delta_{p, \, n - p + 1}\}$.

Let $\delta_{p, \, i}$ be a convex vital element of height $p$ in $G(p)$, as expressed in \eqref{opt} such that $p + 1 \leq i \leq n - p$. It follows from Lemma \ref{prod} that  $\delta_{p, \, i}$ is a product of a convex vital element and an idempotent element, each of height $p + 1$. Notice that the extreme convex vital elements $\delta_{p, \, p}$ and $\delta_{p, \, n - p + 1}$ in $G(p)$, which by Remark \ref{phf} are not expressible as product of elements in $G(p + 1)$, as such, every convex vital element in $G(p)$ can be expressed as a product of elements in $G(p + 1) \cup \{\delta_{p, \, p}, \, \delta_{p, \, n - p + 1}\}$. We have now completed the proof of the lemma.
\end{proof}

\begin{remark}\label{extr}
It is now clear from the above lemma that for $2 \leq p \leq \lceil \frac{n}{2} \rceil$, any minimum generating set of $I(n, \,p)$ must contain all the extreme elements of height below $p$, i.e., the elements $\delta_{p, \, i}$ and $\delta_{p, \, n-i+1}$ for all $2 \leq i \leq p - 1$.\end{remark}
 Now, in $I(n, \,p)$  let \begin{equation}\label{w}W(p)=\left\{
                                                                                 \begin{array}{lll}
                                                                                   E(J^{*}_{p}), & \hbox{if   $p=0,1$;} \\
                                                                                    E(J^{*}_{p})\cup M(\lceil \frac{n}{2} \rceil)\cup\{\delta_{p, \, i}, \, \,  \delta_{p, \, n-i+1}: \, 2\leq i\le \lceil \frac{n}{2} \rceil-1\}, & \hbox{if   $p \in \{\lceil \frac{n}{2} \rceil + 1, \ldots, n-1\}$;} \\
                                                                                   G(p)\cup\{\delta_{p, \, i},  \, \,  \delta_{p, \, n-i+1}: \, 2\leq i\le p-1\}, & \hbox{if $2 \leq p \leq \lceil \frac{n}{2} \rceil$.}
                                                                                 \end{array}
                                                                               \right.
 \end{equation} \noindent Thus, we have the following lemmas.

\begin{lemma}\label{mminknp} For $0\leq p\le n-1$,
$W(p)$ is the minimum generating set of $I(n, \,p)$.
\end{lemma}
\begin{proof} The result follows from the fact that $G(p)$ is the minimum generating set of $\langle J^{*}_{p} \rangle$ by Lemma \ref{minnn}, and also from the fact that each minimum generating set of the ideal $I(n, \,p)$ must contain all the extreme convex vital elements below height $p$, as stated in Remark \ref{extr}.
\end{proof}
\begin{lemma}\label{ooknp} Let $W(p)$ be as defined in \eqref{w}. Then  \[|W(p)|=\left\{
                                                                                 \begin{array}{lll}
                                                                                 1, & \hbox{if   $p=0$;} \\
                                                                                   2^{n}-1, & \hbox{if   $p=1$;} \\
                                                                                    (n-2)+\sum\limits_{r=p}^{n}{\binom{n}{r}}{\binom{r-1}{p-1}}, & \hbox{if   $p \in \{2, \ldots, n-1\}$.}
                                                                                 \end{array}
                                                                               \right.\]
\end{lemma}
\begin{proof} For the cases $p=0$ and $p=1$, the result is trivial. Next, if  $2 \leq p \leq \lceil \frac{n}{2} \rceil$, then $|W(p)|=|G(p)|+|\{\delta_{p, \, i}, \, \,  \delta_{p, \, n-i+1}: \, 2\leq i\le p-1\}|$. Thus, by Lemma \ref{nid2}, we see that $$|W(p)|=(n-2p+2)+\sum\limits_{r=p}^{n}{\binom{n}{r}}{\binom{r-1}{p-1}}+2(p-2)=(n-2)+\sum\limits_{r=p}^{n}{\binom{n}{r}}{\binom{r-1}{p-1}}.$$

\noindent Finally, if $ \lceil \frac{n}{2} \rceil+1 \leq p  \leq n-1$, we see that $|W(p)|=|E(J^{*}_{p})|+|M(\lceil \frac{n}{2} \rceil)|+|\{\delta_{p, \, i},  \, \, \delta_{p, \, n-i+1}: \, 2\leq i\le \lceil \frac{n}{2} \rceil-1\}|$. Thus, by Lemma \ref{nid2}, it follows that
$$|W(p)|=\sum\limits_{r=p}^{n}{\binom{n}{r}}{\binom{r-1}{p-1}}+(n-2\left\lceil \frac{n}{2} \right\rceil+2)+2\left(\left\lceil \frac{n}{2} \right\rceil-2\right)=(n-2)+\sum\limits_{r=p}^{n}{\binom{n}{r}}{\binom{r-1}{p-1}}.$$
The result now follows.
\end{proof}

Consequently, we have the following result.

\begin{theorem}\label{knp} Let $I(n, \,p)$ be as defined in \eqref{kn}. Then,  for $1 \leq p \leq n-1$, we have  \[\text{rank }I(n, \,p)=\left\{
                                                                                 \begin{array}{lll}
                                                                                 1, & \hbox{if   $p=0$;} \\
                                                                                   2^{n}-1, & \hbox{if   $p=1$;} \\
                                                                                    (n-2)+\sum\limits_{r=p}^{n}{\binom{n}{r}}{\binom{r-1}{p-1}}, & \hbox{if   $p \in \{2, \ldots, n-1\}$.}
                                                                                 \end{array}
                                                                               \right.\]
\end{theorem}

\begin{proof} It is clear that $\langle E(J^{*}_{0})\rangle=I(n, \, 0)$ and $\langle E(J^{*}_{1})\rangle=I(n,\, 1)$. Thus, it follows from  Lemmas ~\ref{mminknp} and~\ref{ooknp} that  $\text{rank }I(n,\, 0)=1$ and   $\text{rank }I(n,\, 1)=2^{n}-1$.
\vspace{0.2cm}

\noindent Now, for all $2 \leq p \leq n - 1$, observe that by Lemma \ref{lm1}, $\langle J^{*}_{p} \cup \{\delta_{p, \, i},  \, \, \delta_{p, \, n-i+1} : 2 \leq i \leq p - 1\} \rangle = I(n, \,p)$ . Notice that if $2 \leq p \leq \lceil \frac{n}{2} \rceil$, then \[\langle W(p) \rangle=\langle G(p) \cup \{\delta_{p, \, i},  \, \, \delta_{p, \, n-i+1} : 2\leq i\leq p - 1\}\rangle =\langle J^{*}_{p} 	\cup\{\delta_{p, \, i}, \, \, \delta_{p, \, n-i+1}: 2\leq i\leq p-1\}\rangle;\] \noindent and if $\lceil\frac{n}{2}\rceil + 1\leq p\leq n-1$, then \[\langle W(p) \rangle=\langle E(J^{*}_{p})\cup M(\lceil\frac{n}{2}\rceil)\cup\{\delta_{p, \, i}, \, \, \delta_{p, \, n-i+1}: 2\leq i\leq\lceil\frac{n}{2}\rceil-1\}\rangle =\langle J^{*}_{p}\cup M(\lceil\frac{n}{2}\rceil)\cup\{\delta_{p, \, i}, \, \, \delta_{p, \, n-i+1}: 2\leq i\leq\lceil\frac{n}{2}\rceil-1\}\rangle.\] In either case, the result follows from Lemmas~\ref{mminknp} and~\ref{ooknp}.
\end{proof}
 We now present the following result.
\begin{theorem} Let $\mathcal{DORP}_{n}$ be as defined in \eqref{qn1}. Then,  $\text{ rank }\mathcal{DORP}_{n}=3n-2$.
\end{theorem}
\begin{proof}
Observe that
\begin{align*} \text{ rank }\mathcal{DORP}_{n}&= \text{ rank }I(n,n-1)+ |\{\text{id}_{[n]}\}|\\&=(n-2)+ \sum\limits_{r=n-1}^{n}{\binom{n}{r}}{\binom{r-1}{n-2}}+1 \\&=3n-2,
\end{align*}
as required.
\end{proof}

\vspace{2cm}

\noindent{\bf Data Availability Statement:}\\ Data sharing is not applicable to this article as no data were created or analysed in this study.
\vspace{0.5cm}



\noindent{\bf Conflict of interest:}\\ The results in this paper were obtained in October 2024 and the first version was initially submitted for publication to Semigroup Forum in December 2024. The authors declare that there are no conflicts of interest regarding the publication of this paper. \\
\vspace{0.5cm}


\noindent{\bf Acknowledgements:}\\
 The lead author expresses gratitude to Bayero University and TETFund (TETF/ES/UNI/KANO/TSAS/2022) for providing financial support. Additionally, the author thanks Sultan Qaboos University, Oman, for their hospitality during a postdoctoral research visit from July 2024 to June 2025.

\end{document}